\documentclass[11pt,
]{amsart}
\usepackage{
amssymb, graphicx, color, cite 
}

\usepackage{mathtools}
\mathtoolsset{showonlyrefs}
\usepackage[linkcolor=blue, citecolor=black, breaklinks=true]{hyperref}

\date{\today}

\numberwithin{equation}{section}%

\newtheorem{theorem}{Theorem}[section]

\newtheorem{lemma}{Lemma}[section]
\newtheorem{definition}{Definition}[section]

\theoremstyle{definition}

\newcommand{\eps}{\varepsilon}

\newcommand{\R}{{\mathbb R}}

\newcommand{\Id}{\mbox{Id}}
\renewcommand{\r}[1]{\eqref{#1}}

\newcommand{\be}[1]{\begin{equation}\label{#1}}
\newcommand{\ee}{\end{equation}}

\renewcommand{\d}{\mathrm{d}}

\newcommand{\bo}{\partial M}

\title[Lorentzian scattering rigidity; boundary determination]{Boundary determination and local rigidity of analytic metrics in the Lorentzian scattering rigidity problem}

\author[P. Stefanov]{Plamen Stefanov}
\address{Department of Mathematics, Purdue University, West Lafayette, IN 47907}
\thanks{P.S. partly supported by  NSF  Grant DMS-2154489}

\begin{document}
\begin{abstract}
We study the scattering rigidity problem in Lorentzian geometry: recovery of a Lorentz\-ian metric from the scattering relation  known on a lateral timelike boundary.  We show that one can recover the jet of the metric up to a gauge transformation near a lightlike strictly convex point. Assuming that the metric is real analytic, we show that one can recover the metric up to a gauge transformation as well near such a point. 
\end{abstract} 
\maketitle

\section{Introduction}  
Let $(M,g)$ be a Lorentzian manifold  of dimension $1+n$ with a cylindrical-like timelike boundary, generalizing $\R\times N$ where $N$ is a compact Riemannian manifold with a smooth boundary. 
We define the (lightlike) scattering relations $\mathcal{S}$, and  $\mathcal{S}^\sharp$, acting on vectors and covectors, respectively, as the exit points and directions/codirections on the boundary of lightlike geodesics starting at such points and directions/codirections at the boundary, see Definition~\ref{def1} for a precise statement. The problem we study is to what extent does $\mathcal{S}$ or  $\mathcal{S}^\sharp$ determine $g$. In this generality the problem is wide open with some partial results so far described below. In this paper we show that one can recover the whole jet of $g$ on $\partial M$, assumed strictly convex to light rays, up to a gauge transformation. Also, if $g$  is a priori analytic, one can recover $g$ in $M$ near such points.   

The lightlike  $\mathcal{S}$ probes the metric over a restricted set of geodesics, satisfying $g(\dot\gamma, \dot\gamma)=0$. This takes  one dimension away from the set of all geodesics, and, together with the signature of the metric, makes the Lorentzian version of this problem harder with new   challenges. First, the group of the gauge transformations is richer, since it adds the freedom to multiply by an arbitrary  conformal factor $\mu>0$. In fact, it is arbitrary for $\mathcal{S}^\sharp$ with the additional restriction $\mu=\text{const.}$ on $\bo$ for $\mathcal{S}$. Next, the linearization of this problem is the geodesic light ray transform \cite{S-Lorentzian-scattering}, which is known to  be unable to see timelike singularities; roughly speaking those corresponding to signals moving faster than light. A loss of ellipticity exists in the Riemannian case as well when $n=2$, and we restrict ourselves to ``short'' geodesics close to being tangent to the boundary, but it is of a different nature.  Then it is not a priori intuitively clear in the Lorentzian case  whether one can expect boundary or local/global recovery, the latter even in the analytic case,  but we show that the local recovery is possible.

One possible motivation to study  $\mathcal{S}^\sharp$ comes from the analysis of the wave equation related to $g$. As it was shown in \cite{St-Yang-DN},  $\mathcal{S}^\sharp$ is the canonical relation of the Dirichlet-to-Neumann map $\Lambda$ on $\bo$, which is a Fourier Integral Operator. From relativity point of view, $\mathcal{S}$ and $\mathcal{S}^\sharp$ contains information about the way photon trajectories are affected by the Lorentzian structure of spacetime. A linearization of  $\mathcal{S}^\sharp$ from a spacelike hypersurface (``shortly'' after the Big Bang) to a future one (the present) is studied in \cite{LOSU-strings}, motivated by the information carried by the observed redshift of the cosmic background radiation.

Boundary determination results for Riemannian manifolds with boundary are well known \cite{LassasSU}, including stability estimates \cite{SU-JAMS}, and a constructive algorithm \cite{UhlmannWang}.  The strict convexity condition was relaxed considerably to include  concave points under a certain non-conjugacy condition in \cite{SU-lens}, see also \cite{Zhou-boundary_12}. As a consequence, Vargo \cite{Vargo_09} showed that one can recover an analytic non-trapping Riemannian manifold  up to an isometry  from the lens relation. This result was extended in \cite{Herreros-Vargo} to such manifolds with an analytic magnetic field. 
A major advancement in dimensions $n\ge3$ was done in \cite{SUV_localrigidity,SUV_anisotropic} based on the approach in \cite{UV:local}, where it was shown that for smooth metrics, we can recover not just the jet of $g$ on the boundary up to a gauge, but $g$ inside $M$ as well, near a strictly convex point. This implies global rigidity results under a foliation condition. The approach is based on ellipticity of the linearization (and on the use of Melrose's scattering calculus) which does not hold in our case.

There are many boundary and lens/scattering rigidity results in the Riemannian case, aside from the already mentioned \cite{SUV_localrigidity,SUV_anisotropic} 
see, e.g., \cite{SU-lens,SU-JAMS,SU-MRL,PestovU,MuRo,Mu2,Mu77,Michel,LassasSU,Colin14,CDS,CrokeH02,Croke_scatteringrigidity,Croke04,Burago-Ivanov}. They can be viewed as rigidity results for static Lorentzian metrics.  Very little is known in the Lorentzian case. Recovery of stationary metrics from the time separation function was studied in \cite{UhlmannYangZhou}.  
The author showed in \cite{S-Lorentzian-scattering}, that under some additional conditions, stationary metrics in cylindrical product type of manifolds $M = \R_t\times N_x$ with boundary are lens rigid near a generic set of simple metrics, including simple real  analytic ones. The reason for this is that projected onto the ``base'' $N_x$, such systems reduce to magnetic ones studied before in \cite{St-magnetic}. Scattering rigidity via timelike geodesics was studied in \cite{sebastian2024scattering} for stationary metrics. The dynamical system then projects to a magnetic-potential one studied in \cite{sebastian2024linearization, sebastian2023boundary}. 
Our second main result is that  assuming $(M,g)$ real analytic, it is uniquely determined by  $\mathcal{S}^\sharp$ near strictly lightlike convex boundary points. 

\textbf{Brief description of the approach.} 
To explain the challenges, we note first that we need to prove the existence of a gauge transformation, see Definition~\ref{def_gauge}, relating two metrics $\hat g$ and $g$ with the same data. That means construction of two quantities: a diffeomorphism $\psi$ and a conformal factor $\mu$. In the Riemannian case, we have $\psi$ only, and it is a priori clear what it should be near $\bo$, if there is rigidity: just identify the boundary normal coordinates for $\hat g$ and $g$. Then we need to prove that $\hat g=\psi^* g$ up to infinite order at $\bo$ with $\psi$ known. One way to do this is as in \cite{LassasSU} by looking at the Taylor expansion of $\hat g-\psi^* g$ in the normal variable $x^n$. We do not have such natural candidates for $\psi$ and $\mu$ in the Lorentzian case. To construct such candidates, we ``normalize'' first $\hat g$ and $g$ in the gauge equivalent class, see section~\ref{sec_step1}: to coincide on some timelike field $\partial/\partial x^0$ tangent to $\bo$; and at the same time to be both in boundary normal coordinates, both up to $O((x^n)^\infty)$, as $x^n\to 0+$. This leads to a non-characteristic Cauchy problem for a fully nonlinear PDE, see \r{p5}, which can be transformed easily into a quasilinear one, see \r{nf}. It is not a priori clear that this problem is solvable but it can always be solved up to $O((x^n)^\infty)$. Once we have this, we apply the Taylor series argument using the maximizing property of timelike geodesics. 

Assuming the metrics analytic, one would think that one can just use analytic continuation.  That is essential for this result, of course, but we need to show that $\psi$ and $\mu$ which we constructed only up to an infinite order at $\bo$, actually exist locally. We apply the Cauchy–Kowalevski theorem to show that the ``normalization'' in section~\ref{sec_step1} can be done exactly, locally, allowing us to construct $\psi$ and $\mu$ locally. Then we use analytic continuation. 

One could hopefully prove global rigidity results for analytic metrics under appropriate geometry conditions but that would require some non-trivial efforts, and will be studied in a forthcoming work. 

\textbf{Acknowledgments.} The author thanks Leo Tzou and Lauri Oksanen for an inspiring discussion on related problems, and to Sebasti\'an Mu\~noz-Thon for his critical remarks.

\section{Main results}

We start with the introduction of the main notions, see also \cite{S-Lorentzian-scattering}. 
Let $x\in \bo$, with $\bo$ timelike, and let $0\not=v$ be a lightlike vector at $x$ pointing into $M$. 
Assume that the lightlike geodesic $\gamma_{x,v}(s)$  hits $\bo$ again for the first time for some $s=s(x,v)$, at  $y = \gamma_{x,v}( s) \in \bo$ in the direction $w=\dot \gamma_{x,v}( s) $. 
We can define the scattering relation as the map $(x,v)\mapsto (y,w)$, see Figure~\ref{Lorentz_distance_fig1}.

It is convenient to identify such lightlike vectors $v$ with their orthogonal projections $v'$ to $T\bo$.  Those projections would be timelike, and in the limiting case when $v$ is tangential to $\bo$, they would be lightlike. Recall that  vectors that are either timelike of lightlike are called causal. Then we can think of the scattering relation as defined on the causal cone in $T\bo$ with an image in the causal cone in $T\bo$ unless the corresponding geodesic is trapping.

 This leads to the following. 

\begin{definition} \label{def1}
	The scattering relation $\mathcal{S}$, mapping the causal cone in $T\bo$ to itself, is defined as follows. 
 Let $(x,v')\in T\bo$ be timelike, and let $v$ be the unique lightlike vector at $x$  with orthogonal projection $v'$ on $T_x\bo$, pointing to the interior of $M$.  Then we set   $y\in \bo $ to be the point where the geodesic $\gamma_{x,v}$ issued from $(x,v)$ meets $\bo$ again for the first time, and set  $w'$ to be the orthogonal projection on $T_y\bo$ of its direction there. We set   $\mathcal{S}(x,v') = (y,w')$.
 
 When $(x,v')$ is lightlike, $v=v'$ is tangent to $\bo$, and we set $\mathcal{S}(x,v') = (x,v')$.
 
 	The scattering relation $\mathcal{S}^\sharp$ on the causal cone in $T^*\bo$ is defined as $\mathcal S$ by identifying vectors and covectors by the metric. 
\end{definition}

We use the musical isomorphism notation below converting vectors to covectors and vice versa. 

If $\gamma_{x,v}$ happens to be trapping for some $(x,v)$, we just consider that $(x,v)$ not to be in the domain of $\mathcal{S}$ (and this is not going to be allowed in this paper). Knowing $\mathcal{S}$ or $\mathcal{S}^\sharp$ includes knowing the domain.

\begin{figure}[h!] 
  \centering
  \includegraphics[scale=1,page=4]{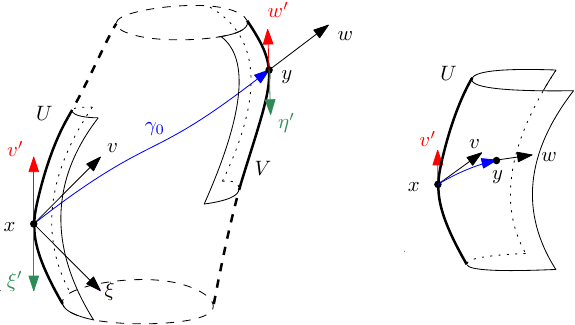}
\caption{\small  The scattering relation $\mathcal{S}(x,v')= (y,w')$ on the causal cone on $T\bo$, left. Its version when $v$ is close to being tangent to $U\subset \bo$, right.
}
\label{Lorentz_distance_fig1}
\end{figure}

 Clearly, $\mathcal S$ and $\mathcal{S}^\sharp$ are positively homogeneous of order one in the fiber variable. We have $\mathcal{S}(x,a(x,v')v') = (y,a(x,v')w') $ for every $a>0$. We may normalize $v'$ in some way to reduce the number of variables. For example,  if $x^0$ is a local time variable, we may require $v^0=1$. 

The definition of $\mathcal{S}$ or $\mathcal{S}^\sharp$ requires us to know which vectors/covectors on the boundary are causal, which is equivalent to knowing a conformal multiple of $g$ restricted to $T\bo\times T\bo$. Instead of prescribing $\mathcal S$ or $\mathcal{S}^\sharp$, we can consider the pairs of points $(x,y)$ in  $\Sigma\subset U\times V\subset \bo\times\bo$ which can be connected by a lightlike geodesic, so that $x$ and $y$ are not conjugate along it. It was shown in  \cite{S-Lorentzian-scattering} that $\Sigma$ and $\mathcal{S}^\sharp$ determine each other uniquely. 

\begin{definition}\label{def_gauge}
We call  $(M,g)$ and $(\hat M, \hat g)$ \emph{gauge equivalent}, if there exists a diffeomorphism $\psi: M\to \hat M$ fixing $\partial M$ pointwise, and a function $\mu>0$, so that $ g   = \mu  \psi^* \hat g$. When studying a local version of $\mathcal{S}$ or $\mathcal{S}^\sharp$, we assume that $\psi$ and $\mu$ are defined in the open set of $M$ covered by the geodesic used in the definition of $\mathcal{S}$.
\end{definition}

Fixing a positive sign of the conformal factors, as we did, preserves the future/past orientation. 
We showed in \cite{S-Lorentzian-scattering} that gauge equivalent Lorentzian metrics on the same manifold have the same $\mathcal{S}^\sharp$. The relation $\mathcal{S}$, on the other hand, is affected by $\mu$ on $\bo$, and gauge transformations preserve $\mathcal{S}$ if  $\mu=\text{const.}$ on $\bo$ only.  
In both cases, the gauge transformations form a group generated by conformal multiples and isometries fixing $\bo$ pointwise.

We define the notion of strict convexity as in \cite{S-support2014}.

\begin{definition}\label{def_convex}
Let $U$ be a  smooth hypersurface near a point $z\in U$ and let $F$ be a defining function so that $U=F^{-1}(0)$ near $z$, $\d F(z)\not=0$, and declare $\{F>0\}$,  to be the ``interior'' of $M$ near $z$. Similarly, $\{F<0\}$ is the ``exterior'' of $M$ near $z$.  We say that $U$ is \textit{strictly convex} at $z$ in the direction $v\in T_zU$, if  $\nabla^2F(z)(v,v)>0$. 
%
\end{definition}

Strict  convexity of $U$ for $(z,v)$ restricted to a set of null vectors can be defined in a similar  way. In this paper, the directions would be lightlike, and we call this strict lightlike convexity. 
Here $\nabla^2F$ is the Hessian of $F$, with $\nabla$ being the covariant derivative. This notion of strict convexity at $(z,v)$ is equivalent to $\frac{\d^2} {\d s^2}\big|_{s=0}F\circ\gamma(s)<0$ for the geodesic $\gamma$ through $z$ in the direction $v$; and it is independent of the choice of $F$. When $U\subset\bo$, the interior is $M$, and the exterior does not exist. We can extend $g$ smoothly on the other side of $M$ however, and then this characterization still holds independently of the extension. 

Our first main result is boundary recovery for smooth metrics.    
   
\begin{theorem}[Boundary Recovery] \label{thm_bdry}
	Let $g$ and $\hat g$ be two Lorentzian metrics defined near some $x_0\in\bo$ so that $\hat g = \mu_0 g$ on $T\bo\times T\bo$ locally  with some $0<\mu_0\in C^\infty(\bo)$. Let $(x_0,v_0)\in T\bo$ be lightlike for $g$. Assume that $\bo$ is strictly convex with respect to   $g$  in the direction of $(x_0,v_0)$.  
	Assume 
	that either 
	\begin{itemize}
		\item[(i)]  $ \hat{\mathcal S}^\sharp = \mathcal{S}^\sharp$ in a neighborhood of $(x_0,v_0^\flat)$, or
		\item[(ii)] $\hat{\mathcal S}= \mathcal{S}  $ in a neighborhood of $(x_0,v_0)$, and $\mu_0=\text{const}$.  
	\end{itemize}
	Then there exists $\mu(x)>0$ with $\mu=\mu_0$ on $\bo$, and a local diffeomorphism $\psi$ near $\bo$ preserving it pointwise, so that the jets of $g$ and $\mu\psi^*\hat g$ coincide on $\bo$ near $x_0$.  
\end{theorem}

Next, we prove local rigidity for analytic manifolds and metrics. Analyticity always means real analyticity in this paper, and analyticity in $M$ always means analyticity up to the boundary, i.e., existence of an analytic extension in some two-sided neighborhood of $M$, where $M$ is assumed to have an analytic extension as well. Similarly, analyticity near $x_0\in \bo$ means analyticity in a two-sided neighborhood of $x_0$ in the extended $M$. 

\begin{theorem}[Local rigidity for analytic metrics]\label{thm_local} 
	Under the assumption of Theorem~\ref{thm_bdry}, assume that $M$, $\bo$ are analytic, $g$, $\hat g$ are analytic as well in $M$ near $x_0$, and $\mu_0$ is analytic on $\bo$ near $x_0$. Assume either (i) or (ii) of Theorem~\ref{thm_bdry}. Then there exists an analytic $\mu(x)>0$ with $\mu=\mu_0$ on $\bo$, and an analytic local diffeomorphism $\psi$ near $x_0$ preserving $\bo$ pointwise, so that $g = \mu\psi^*\hat g$  near $x_0$. 
\end{theorem}

\section{Preliminary results}   \label{sec_prel}

\subsection{Gauge transformations}
We review some results from \cite{S-Lorentzian-scattering}. Recall that multiplying $g$ by a conformal factor $\mu>0$ reparameterizes the future directed null geodesics, keeping the future orientation, but leaves them the same as point sets \cite{LOSU-strings, S-Lorentzian-scattering}. This extends easily to more general Hamiltonians and to $\mu$ depending on both $x$ and $\xi$.  Next lemma is stated in a greater generality than we need it, and it is known in principle. 
\begin{lemma}
	Let $H(x,\xi)$ be a Hamiltonian defined near $(x_0,\xi^0)$, and assume $ H(x_0,\xi^0)=0$, $\d H(x_0,\xi^0)\not=0$. Let $\mu(x,\xi)>0$ near $(x_0,\xi^0)$. Then the Hamiltonian curves of $H$ and $\tilde H:= \mu H$ near $(x_0,\xi^0)$ on the zero energy level coincide as point sets but have possibly different parameterizations. 
	
	More specifically, if $(x(s),\xi(s))$, and $(\tilde x(\tilde s), \tilde \xi(\tilde s))$ are solutions related to $H$ and $\tilde H$, with initial conditions $(x_0,\xi^0)$ at $\tilde s=  s_0$, and $(\tilde x_0,\tilde \xi^0)$ at $s=s_0$, $\tilde s=\tilde s_0$, respectively, then 
	\[
	(\tilde x(\tilde s(s)), \tilde \xi(\tilde s(s))) = (x(s),\xi(s))
	\]
	with $\tilde s(s)$ solving
	\[
	\frac{\d \tilde s}{\d  s}=\mu^{-1}( x( s),\xi(s)), \quad \tilde s(s_0)= \tilde s_0.
	\]
	
\end{lemma}

\begin{proof}
	The Hamiltonian system for $\big(\tilde x(\tilde s),\tilde \xi(\tilde s) \big)$ related to $\tilde H$ reads
	\[
	\frac{\d \tilde x}{\d \tilde s} =  \mu H_\xi+ \mu_\xi H, \quad \frac{\d \tilde \xi}{\d \tilde s} = - \mu H_x- \mu_x H.
	\]
	Assume initial conditions $(x_0,\xi^0)$ at $\tilde s=\tilde s_0$. On the energy level $H=\tilde H=0$, we are left with
	\[
	\frac{\d \tilde x}{\d \tilde s} =  \mu H_\xi, \quad \frac{\d \tilde \xi}{\d \tilde s} = - \mu H_x.
	\]
	Then for $s\mapsto (\tilde x(\tilde s(s)), \tilde \xi(\tilde s(s)))$ we get the Hamiltonian system related to $H$ with initial conditions as stated above. This completes the proof. 
\end{proof}

We apply the lemma to the Hamiltonian $H(x,\xi)= \frac12 g^{ij}(x)\xi_i\xi_j$, written in local coordinates. It is well known that the Hamiltonian curves of $H$ on $T^*M$ at zero energy level, when identified with curves in $TM$ by the musical isomorphism, coincide with the lightlike geodesics. When $\hat g = \mu g$, the corresponding Hamiltonian is $\hat H= \mu^{-1}H$. Then 
\[
(\tilde x(\tilde s), \tilde \xi(\tilde s)) = (x(s(\tilde s)),\xi(s(\tilde s)))
\]
with 
$ s(\tilde s)$ solving
\[
\frac{\d s}{\d \tilde  s}=\mu^{-1}( x(\tilde  s)), \quad  s(\tilde s_0)=  s_0.
\]
This shows that $\mathcal{S}^\sharp$ is invariant under the conformal transformation $\tilde g=\mu g$. Indeed, $(x,\xi)\mapsto (y,\eta)$ is unchanged, see Figure~\ref{Lorentz_distance_fig1}. Next, $\xi$ has the unique decomposition $\xi=\xi'+\xi''$, where $\xi'\in T_x^*U$, and $\xi''$ is conormal to $TU$, i.e., normal to all covectors in $T^*U$ in the metric $g^{-1}$. This decomposition does not change when we replace $g$ by $\mu g$. Therefore, the projection $\xi'$ is independent of the conformal factor, and the same applies to $\eta'$ at $y$. Thus $\mathcal{S}^\sharp$ is independent of a conformal transformation. On the other hand, when $\hat g = \mu_0 g$ on $T\bo \times T\bo$, one can multiply either metric by $\mu_1>0$ with $\mu_1=1$ on $\bo$ only to preserve that condition. 

Those arguments apply to the vectors $v$ and $w$ as well but there is an essential difference. The map $(x,v)\mapsto (y,w)$, before the projections, does depend on $\mu$ because the musical isomorphism converting $\xi'$ into $v'=(\xi')^\sharp$, and similarly for $\eta'$ and $w'$ brings the factor $\mu^{-1}$.  Indeed, assume $\mathcal{S}=\tilde{\mathcal{S}}$ now. Then $\mathcal{S}(x,v)=(y,w)$ can be computed as 
\[
(x,v) \xmapsto{\ \flat\ }(x,gv) \xmapsto{\ \mathcal{S}^\sharp \ } (y,\eta) \xmapsto{\ \sharp\ } (y,w=g^{-1}\eta),
\]
where $(y,\eta) = S^\sharp(x,gv)$ by definition. 
With $\tilde g = \mu g$, we have $\tilde{\mathcal{S}}^\sharp = {\mathcal{S}}^\sharp$, and  $\tilde{\mathcal{S}}(x,v)=(y,\tilde w)$ can be computed as 
\[
(x,v) \xmapsto{\ \flat\ }(x,\mu(x)gv) \xmapsto{\tilde{\mathcal{S}}^\sharp= \mathcal{S}^\sharp} (y,\mu(x)\eta) \xmapsto{\ \sharp\ } \left(y,\tilde w =\mu^{-1}(y) g^{-1}(y)\mu(x)\eta\right) .
\]
Therefore,
\[
\mathcal{S}(x,v)=(y,w) \quad \Longrightarrow  \quad 
\tilde{\mathcal{S}}(x,v) = (y, \mu^{-1}(y)  \mu(x)w),
\]
see also  in \cite[eq.~(17)]{LOSU-strings}. 
Therefore, $\mathcal{S}$ is preserved under the conformal change $\tilde g = \mu g$ if and only if $\mu(x)=\mu(y)$ for all $(x,y)\in\Sigma$, which implies $\mu=\text{const.}$ on $U\times V$. Moreover, if $\hat g = g$ on $T\bo\times T\bo$, then the condition is $\mu=1$. Under the conditions of Theorem~\ref{thm_bdry}, when $\hat g=\mu_0 g$ on $T\bo\times \bo$, and they have the same scattering relation, the conformal freedom we have is to multiply one of the metrics by $\mu_1>0$ with $\mu_1=1$ on $\bo$ which preserves the boundary condition, and it is the conformal factor allowable to preserve $\mathcal{S}$. Then we get the same condition as that for $\mathcal{S}^\sharp$, see the last sentence of the theorem, but for two reasons, instead for one. 

Note however that $\mathcal{S}_{\tilde g} = \mathcal{S}_g$ for two metrics, implies $\mathcal{S}_{\mu\tilde g} = \mathcal{S}_{\mu g}$ for every $\mu>0$. 

On the other hand, any diffeomorphism $\psi$ with $\psi=\Id$ on $\bo$ preserves both $\mathcal{S}$ and $\mathcal{S}^\sharp$. It intertwines with multiplying $g$ by a conformal factor without changing its boundary values. 

\subsection{Jets of tensors fields} \label{sec_3.2}
We fix a small neighborhood $U\subset \bo$ of $x_0\in\bo$, reserving the right to shrink it several times during the proof. We use coordinates $x=(x',x^n)$, $x'=(x^0,x^1,\dots, x^{n-1})$ in $M$ near $x_0=0$, where  $U\subset \bo$ is given locally by   $x^n=0$ with $x^n\ge0$ in $M$. 
Given a tensor field $f$ of type $(0,2)$ (metric-like), the jet of $f$ on the hypersurface $x^n=0$ in such a local chart is given by $\partial_{x^n}^k f_{ij}$ on $x^n=0$ for all $k=0,1,\dots$. Knowing the jet of $f$ implies knowing all derivatives of $f$ at $x^n=0$, not just the $x^n$ ones. Under a change of variables, aside from having new variables,  $f$ changes its coordinate representation as well. The following lemma shows that knowing the jet is a  coordinate independent notion. 

\begin{lemma}
	Let $f$ and $\hat f$ be two type $(0,2)$ tensor fields with the same jets at $x^n=0$ near $x=0$ in some local coordinates. Let $x=\psi(y)$ be a local diffeomorphism near the origin so that $y^n=0$ is mapped to $x^n=0$. Then $\psi^*f$ and $\psi^*\hat f$ have the same jets on $y^n=0$ in the $y$ variables. 
\end{lemma}

\begin{proof}
	It is enough to assume that $\hat f=0$. 
	We have 
	\[
	(\psi^*f)_{ij}(y) = f_{i'j'}(x(y) )\frac{\partial x^{i'}}{\partial  y^i} \frac{ \partial  x^{j'}}{\partial  y^j}.
	\]
	On $\{x^n=0\} = \{y^n=0\} $ we have $f=0$, which implies the same for $\psi^*f$. More generally, applying any, say constant coefficient (in the $y$ variables) differential operator $P$ to the left-hand side, yields zero on the right on the same hypersurface since the jet of $f$ is zero there. 
\end{proof}

In particular, we can take $\psi$ fixing $x^n=0$ pointwise locally, i.e.,   $x^\alpha=y^\alpha$, $\alpha\le n-1$, and $x^n=x^n(y)$. Then two tensor fields having the same jet at $\{x^n=0\}$ is equivalent to having the same same $y^n$ derivatives of every order under any such change. 

\subsection{Boundary normal coordinates} 
   
Lorentzian manifolds with boundary admit boundary normal coordinates similarly to Riemannian ones.  The following lemma is formulated in \cite{St-Yang-DN} but the proof is in    \cite{Petrov_book}. It  is based on the fact that the lines $x'=\text{const.}$, $x^n=s$ are unit speed geodesics; therefore the Christoffel symbols  $\Gamma_{nn}^i$ vanish for all $ i$.  
   
\begin{lemma} \label{seminormal}
Let $S$ be  a timelike  hypersurface in $M$. For every $x_0\in S$, there exist $\eps>0$, a neighborhood $W$ of $x_0$  in $M$, and a diffeomorphism $\Psi:S\cap W \times [0,\eps)\rightarrow W$ such that

 (i)  $\Psi(z,0)=z$ for all $z\in S\cap W$;
 
(ii)  $\Psi(z,x^n) = \gamma_{z,\nu}(x^n)$ where $ \gamma_{z,\nu}(x^n)$ is the unit speed geodesic issued from $z$ normal to $S$.

Moreover,   if $x'=(x^0,\dots,x^{n-1})$ are local boundary coordinates on $S$, in the coordinate system \allowbreak $(x^0,\dots,x^n)$,  the metric tensor $g$ takes the form
\begin{equation} \label{metricform}
g=g_{\alpha\beta}(x) \d x^\alpha\otimes \d x^\beta + \d x^n\otimes \d x^n, \quad 0\le \alpha,\beta\le n-1. 
\end{equation}
\end{lemma}
 
Clearly, $g_{\alpha\beta}$ has a Lorentzian signature as well.  If $M$ has a boundary, then $S$ can be $\bo$ and $x^n$ is restricted to $[0,\eps]$.  
 We will call such coordinates  the boundary normal coordinates. The lemma remains true if $S$ is spacelike with a negative sign in front of $\d x^n\otimes \d x^n$ in \r{metricform} (we replace the index $n$ by $0$ below), and this gives us a way to define a time function $t=x^0$ locally, and put the metric in the block form  
\begin{equation} \label{metricform2}
g=-\d t^2 + g_{ij}(t,x) \d x^i\otimes \d x^j, \quad 1\le i,j\le n
\end{equation}
with $g_{ij}$ Riemannian.

\subsection{Tensor fields vanishing on the lightlike cone} 
\begin{lemma}\label{lemma_conf}
	Let $h = \{h_{ij}\}$ be a tensor such that $h_{ij}v^i v^j=0$ for all $v$ lightlike for the metric $g$, with both $h$ and $g$ constant. Then $h=cg$ with some constant $c$. 
\end{lemma}

\begin{proof}
Applying a Lorentzian transformation, we can always assume that $g$ is Minkowski. In fact, such a transformation would produce the Minkowski metric times a conformal factor.  	Take $v=(1,0, \dots,0, \sin\alpha,\cos\alpha)$. Expanding $h_{ij}v^i v^j=0$ in Fourier series in $\alpha$, we get
\begin{align*}
h_{00}&+ 2h_{0\, n-1}\sin\alpha + 2h_{0 n}\cos\alpha\\
& + \frac12\left(1-\cos(2\alpha)\right) h_{n-1\,n-1} + \sin(2\alpha)h_{n\,n-1} +\frac12 (1+\cos(2\alpha)) h_{n n}=0.
\end{align*}	
This implies $-h_{00} = h_{nn} = h_{n-1\,n-1}$, $h_{0\,n-1} = h_{0\,n} = h_{n\,n-1}=0$, i.e., the $\{0,n-1,n\}\times \{0,n-1,n\}$ block of $h$ is conformal to the $1+2$ Minkowski one, with conformal factor -$h_{00}$. We can put $\sin\alpha$ and $\cos\alpha$ in any two different positions different from the zeroth one to complete the proof. 

An alternative proof is to identify the conformal factor $c$ as $c=-h_{00}$ first, if $g$ is Minkowski. Then for $f:= h-cg$,  we get
\[
f_{ij}v^i v^j+ f_{i 0}v^i=0, \quad 1\le i,j\le n
\]
for all $\{v^i\}_{i=1}^{n}$ unit in the Euclidean norm. By \cite[Lemma~3.3]{St-magnetic}, this implies $f=0$.  
\end{proof}

As a consequence, if $h$ and $g$ depend on $x$, and  $h_{ij}(x)v^i v^j=0$ for all lightlike $(x,v)$ near a fixed $(x_0,v_0)$, we can extend this to all $v$ based at those points by analyticity on $v$ (and the proof is local in $v$ anyway), and we get the conclusion in the lemma with $c=c(x)$ having the regularity of $g$ and $h$: smooth or analytic.

\section{Recovery of the jet at the boundary} \label{sec_jet}
We prove Theorem~\ref{thm_bdry} in this section. 

\subsection{Step 1: ``Normalizing'' $\hat g$ and $g$} \label{sec_step1} 
Recall the coordinate convention of section~\ref{sec_3.2}. 
We assume that $x^0$ (not to be confused with the point $x_0=0$) is a local time coordinate, i.e, $g(\partial/\partial x^0, \partial/\partial x^0)<0$. 
We will choose a  metric $\psi^*(\mu\hat g)$, gauge equivalent to $\hat g$ such that it  has the form \r{metricform} up to $O((x^n)^\infty)$ with respect to the same boundary normal coordinates related to $g$, and it coincides with $g$ on $\partial/\partial x^0\times \partial/\partial x^0$ in $M$ up to $O((x^n)^\infty)$. We use the freedom to make conformal changes at this step but we do not use the equality of the scattering relations yet. 
Note that when the metrics are of the form $\mu (-\d t^2+h(x,\d x))$ (the conformally Riemannian case), this can always be achieved by putting $h$ in boundary local coordinates. 

We  normalize   $g$ and $\hat g$ first conformally assuming
\begin{equation}   \label{1'}
	\hat g = g\  \text{on $T\bo\times T\bo$, i.e.}, \quad  	\hat g_{\alpha\beta}(x',0) = g_{\alpha\beta}(x',0). 
\end{equation}
This can be easily achieved by dividing 
$\hat g$ by  $ \mu_0 >0$ with $\mu_0$ extended in $M$ near $x_0$. 
This does not change the data in either case, $\mathcal{S}$ or $\mathcal{S}^\sharp$ and allows us to assume $\mu_0=1$. 
Then we pass to \r{metricform} with respect to $g$, which does not affect \r{1'}. We do the same for $\hat g$, assuming at this point that $(x',x^n)$ are common boundary normal coordinates for both $g$ and $\hat g$. 

We want $\mu$ and $\psi$ to solve
\begin{equation}   \label{p1}
	[\psi^*(\mu\hat g)]_{00} = g_{00}, \quad [\psi^*(\mu\hat g)]_{in} = \delta_{in}, \quad i=0,1,\dots,n
\end{equation}
up to $O((x^n)^\infty)$ with boundary conditions $\psi=\Id$ on $x^n=0$, $\mu=1$ on $x^n=0$. The conversion to common boundary normal coordinates guarantees \r{p1} with $\mu=1$ on $x^n=0$ only. 
From now on, we assume that Greek indices run from $0$ to $n-1$ while Latin ones run from $0$ to $n$. Equation~\r{p1} is equivalent to
\begin{align}\label{p2}
	\mu(\psi(x)) \hat g_{\alpha\beta}(\psi(x))\frac{\partial \psi^\alpha}{\partial x^0} \frac{\partial \psi^\beta}{\partial x^0} &= g_{00},\\
	\mu(\psi(x))\Big( \hat g_{\alpha \beta}(\psi(x))\frac{\partial \psi^\alpha}{\partial x^i} \frac{\partial \psi^\beta}{\partial x^n} +   \frac{\partial \psi^n}{\partial x^i}  \frac{\partial \psi^n}{\partial x^n}  \Big)& = \delta_{in}, \quad i=0,1,\dots,n
	 \label{p3}
\end{align}
up to $O((x^n)^\infty)$ again. 
This is an $(n+2)\times(n+2)$ system for $(\psi,\mu)$ but there are no derivatives of $\mu$ involved. The boundary condition is 
\begin{equation}   \label{p4}
	\psi(x',0)= (x',0), \quad \mu(x',0)=1.
\end{equation}
The latter condition is automatically satisfied if the former is, as a consequence of \r{p2} and \r{1'}. 
We can eliminate $\mu$ in \r{p3} and \r{p2}   to get the fully nonlinear first order  $(n+1)\times(n+1)$ system  
\begin{equation}\label{p5}
	 \hat g_{\alpha \beta}(\psi(x))\frac{\partial \psi^\alpha}{\partial x^i} \frac{\partial \psi^\beta}{\partial x^n} +   \frac{\partial \psi^n}{\partial x^i}  \frac{\partial \psi^n}{\partial x^n}    =  \frac{\delta_{in} }{g_{00}} \hat g_{\alpha\beta}(\psi(x))\frac{\partial \psi^\alpha}{\partial x^0} \frac{\partial \psi^\beta}{\partial x^0} , \quad i=0,1,\dots,n  
\end{equation}
for $\psi$, equivalent to  \r{p3} by \r{p2} since having $\psi$, we can solve for $\mu$ in \r{p2}. 
Equation \r{p5} can also be written as 
\be{p5a}
 (\hat g\circ\psi) (\partial_{x^i}\psi, \partial_{x^n} \psi)=  \frac{\delta_{in} }{g_{00}} (\hat g\circ\psi) (\partial_{x^0}\psi, \partial_{x^0} \psi) , \quad i=0,1,\dots,n .
\ee

We will prove that the boundary $x^n=0$ is non-characteristic for \r{p5}, \r{p4}. More precisely, we show that 
under the additional assumption $\partial\psi^n/\partial x^n>0$ at $x^n=0$, we can put \r{p5} in normal form, see also \r{nf} below. 
Differentiating \r{p4} with respect to $x'$, we get $\partial\psi^i/\partial x^\alpha =\delta^i_\alpha $, $i=0,\dots,n$, $\alpha = 0,\dots,n-1$ at $x^n=0$. In particular, the right-hand side of \r{p5} equals $\delta_{in}\hat g_{00}(x',0)/g_{00}(x',0)=\delta_{in}$.  
We can write \r{p5}, equivalently, \r{p5a}, at $x^n=0$ as $ \hat g(\partial_{x^i}\psi, \partial_{x^n} \psi)= \delta_{in}$, which means that $\partial_{x^n} \psi$ is orthogonal to the basis vectors $e_\alpha$, $0\le \alpha\le n-1$, and unit. This can be solved for $\partial_{x^n} \psi$ at $x^n=0$ 
implying $\partial_{x^n} \psi=e_n$ at $x^n=0$.   Equation \r{p5a} can be solved for $\partial_{x^n} \psi$ when $x$ is near $x_0=0$ as well, giving us the normal form, see \r{nf} again. 
We get 
\begin{equation}   \label{psi}
	\psi = \Id+ O((x^n)^2).
\end{equation} 
In particular, $\psi$ is a diffeomorphism near $x_0$.

Therefore, we can find a solution of \r{p1}, \r{p4} up to $O((x^n)^\infty)$. Note that the usual argument is that if a solution exists, we can compute the Taylor expansion at the boundary. The same arguments show that an asymptotic solution actually exists, even if we do not know that an exact one does (and this argument is used in the proof of the Cauchy Kowalevsky theorem). 

We replace $\hat g$ by the gauge equivalent $\psi^*(\mu \hat g)$.  Then we have 
\begin{equation}   \label{14}
\hat g_{00} = g_{00}+ O((x^n)^\infty), \quad \hat g_{in} = g_{in}+ O((x^n)^\infty), \quad i=0,\dots,n.
\end{equation}
Recall the boundary conditions \r{1'}. 
We want to prove next that $\hat S = S$ implies
\begin{equation}   \label{15}
	\hat g = g+ O((x^n)^\infty).
\end{equation}
for $x\in M$ near $x_0$.

\subsection{Step 2: $\bo$ is strictly convex with respect to $\hat g$ as well} 
The geodesic equation implies
\begin{equation}\label{2}
\ddot \gamma^n+ \Gamma^n_{\alpha\beta}\dot\gamma^\alpha \dot\gamma^\beta=0.
\end{equation}
Recall that $0\le \alpha,\beta\le n-1$ which is implied by the property $\Gamma^n_{n j}=0$, $\forall j$. 
In the coordinates in Lemma~\ref{seminormal}, 
$\Gamma^n_{\alpha\beta}=-\frac12 \partial_{x^n} g_{\alpha\beta}$, which is also the second fundamental form of $U$. Strict convexity at $(x,v)\in TU\setminus 0$ is equivalent to $\Gamma^n_{\alpha\beta}(x) v^\alpha v^\beta>0$.
It is positive definite along  vectors close to $(x_0,v_0)$  by assumption. We can assume that the metric $g$ is Minkowski at $x_0=(0,0)$. 
We take  $v=v_{\eps}$ lightlike,  pointing to the interior of $U$, so that it 
converges to $v_0$, as $\eps\to0+$. We do this by taking $v=v_{\eps}=(1,\sqrt{1-\eps^2 }\theta  ,\eps)$ with $\theta$ unit in $\R^{n-1}$; then $v_0=(1,\theta,0)$. Then for the $n$-th component of the  geodesic $\gamma= \gamma_{0, v_{\eps}}(s)$ through $x_0=0$, with initial direction $\dot\gamma(0)=v_{\eps}$,  we have
%
\begin{equation}\label{3}
\gamma^n(s) =  s\eps - \frac{s^2}2 \Gamma^n_{\alpha\beta}(0) v_{\eps}^\alpha v_{\eps}^\beta +O(s^3 ).
\end{equation}
%
 The function $\gamma^n(s) /s$ then has a non-negative zero for $s= \tau (v_{\eps}):=2\eps/ \Gamma^n_{\alpha\beta}(0) v^\alpha_0 v^\beta_0 +O(\eps^2) $. We can think of it as an escape ``time.'' Then $\gamma_{0,v_{\eps}}(\tau (v_{\eps})s)= \gamma_{0,\tau (v_{\eps} ) v_{\eps} }(s)$ reaches $U$ again for $s=1$.

Let $y_\eps\in U$ be the first intersection point of the geodesic $\gamma$ with $U$ (not counting the initial point $x_0$). 
Then 
\begin{equation}\label{4}
w_\eps:= \exp_x^{-1}y_\eps = \tau (v_{\eps}) v_{\eps} = 2\eps v_0/ \Gamma^n_{\alpha\beta}(0) v^\alpha_0 v^\beta_0 +O(\eps^2).
\end{equation}
 Since $\d\exp_x$ is identity at the origin, we get the same asymptotic expansion for $y_\eps$ in the so chosen local coordinates. 

We will show that $U$ is strictly convex at $(x_0,v_0)$ with respect to $\hat g$ as well. Since $\hat g=g$  on $T\bo\times T\bo$ by \r{1'}, we know that $(x_0,v_0)$ is lightlike for $\hat g$ as well. 
By \r{14}, the lightlike vector at $x_0$ pointing into $M$, related to $\hat g$, is still $v_0$. Then \r{3} still holds with $ \Gamma^n_{\alpha\beta}$ replaced by  $\hat\Gamma^n_{\alpha\beta}$ with the hat over a quantity indicating that it is related to $\hat g$. Note that $(x',x^n)$ are (exact) boundary normal coordinates for $g$ only, and only such up to $O((x^n)^\infty)$ with respect to $\hat g$ by \r{14} but this does not affect our argument.  
Since 
$\hat{\mathcal{S}}=\mathcal S$ or $\hat{\mathcal{S}}^\sharp=\mathcal S^\sharp$, 
$\hat\gamma^n(s)/s$ still has a positive zero at $s\sim \eps$, which implies $\hat\Gamma^n_{\alpha\beta}(0)>0$; thus $U$ is strictly convex at $(x_0,v_0)$, and therefore near it as well,  with respect to $\hat g$ as well.

\subsection{Step 3: recovery of the jet of the metric} 
In preparation for the final step, assume that \r{15} does not hold. Then $\hat g-g  = (x^n)^kh+ O((x^n)^{k+1})$ near $x_0$ with some symmetric tensor field $h=h(x')$ not vanishing at $x_0=0$, and with some $k=1,2,\dots$. By \r{14}, $h_{00}=0$, $h_{ni}=0$, $\forall i$. By Lemma~\ref{lemma_conf}, if we prove that   
$h_{\alpha\beta}(0)  v^\alpha v^\beta=0$ for all $v\in T_0\bo$, close to $v_0$, lightlike for $g$ on $T\bo\times T\bo$,   then we would get $h(0)=\lambda g(0)$ on $T\bo\times T\bo$ with some $\lambda$. Then $0=h_{00}(0)=\lambda g_{00}(0)$ by \r{p1}; therefore $\lambda=0$, and then $h(0)=0$. That would be a contradiction.

Assume $h_{\alpha\beta}(0,0)v_0^\alpha v_0^\beta\not=0$. 
Without loss of generality, we can assume $h_{\alpha\beta}(0,0)v_0^\alpha v_0^\beta>0$; if the opposite, we can switch $g$ and $\hat g$ below. Then $(\hat g -g)(x)v^\alpha v^\beta>0$ for $x$ in the interior of $M$ (but not on $U$) close enough to $x=0$, and for $v$ close enough to $v_0$. With $y_\eps$ as above, let $\gamma_\eps$ and $\hat\gamma_\eps$ be the lightlike geodesics in the metrics $g$ and $\hat g$, respectively, connecting $x=0$ and $y_\eps$. By what we assumed, $\hat\gamma_\eps$ is a timelike curve in the metric $g$ when $0<\eps\ll1$, except for the endpoints, where it is lightlike. Assume it is parameterized by $s\in[0,1]$. Consider the points $\hat\gamma_\eps(\delta)$, $\hat\gamma_\eps(1-\delta)$ with $0<\delta<1/2$. Those two points are connected by the timelike $\hat\gamma_\eps$ (for $g$), see Figure~\ref{Lorentz_distance_fig2} where the dashed curves are geodesic segments related to $g$. 
\begin{figure}[h!] 
  \centering
  \includegraphics[scale=1,page=2]{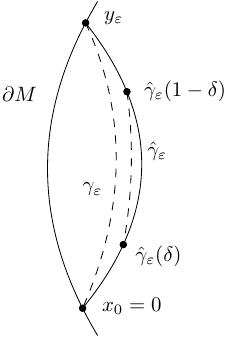}
\caption{\small  Illustration to Step~3.
}
\label{Lorentz_distance_fig2}
\end{figure} 
 By \cite[Proposition~5.34]{ONeill-book}, the ``radial'' geodesic segment (in the metric $g$, again)  between those two points is the unique longest timelike curve connecting them when   $0<\eps\ll1$. 
We recall that the radial geodesic connecting $a$ and $b$ is the geodesic from $a$ with direction $\exp_x^{-1}y$, assuming the inverse exponential map is well-defined; and we restrict the considerations close to $x_0$. Passing to the limit $\delta\to0+$, we get that the length of $\hat\gamma_\eps$ has an upper limit $0$ because the unique timelike geodesic in the proposition tends to $\gamma_\eps$, which is lightlike. 
Thus we get that $\hat\gamma_\eps$  is lightlike for $g$ as well; in particular $\hat\gamma_\eps$ and $\gamma_\eps$ coincide as point sets. This contradicts our assumption. We can perturb $v_0$ on $T_0\bo$ a bit keeping it lightlike for $g$. Perturbing $x_0=0$ now, we get that $h=0$ on $TU\times TU$. Then we get the needed contradiction.

This completes the proof of  Theorem~\ref{thm_bdry}.

\section{Local rigidity for analytic metrics} 
\begin{proof}[Proof of Theorem~\ref{thm_local}]
We return to the system \r{p2}, \r{p3}, \r{p4}, this time with the analyticity assumptions. As explained in section~\ref{sec_step1},  the system can be reduced to \r{p5},  \r{p4}, and after solving for $\psi$, we determine $\mu$ from \r{p2}. The reduced system has analytic coefficients near $x_0=0$ when the original system is analytic, including $\mu_0$ on $\bo$ extended analytically to $M$ near $x_0$. 
To make the boundary condition homogeneous, we set $\phi(x):= \psi(x)-x$. Then, see also section~\ref{sec_step1}, we recast \r{p5a} as
\be{p5.1a}
 (\hat g\circ(\phi+\Id)) (\partial_{x^i}\phi+e_i, \partial_{x^n} \phi +e_n)=  \frac{\delta_{in} }{g_{00}} (\hat g\circ(\phi+\Id)) (\partial_{x^0}\phi+ e_0, \partial_{x^0}\phi+e_0 ) , \quad i=0,1,\dots,n 
\ee
with Cauchy data  
\begin{equation}   \label{5.2}
	\phi(x',0)=0.
\end{equation}
As we showed in section~\ref{sec_step1}, \r{p5.1a}, \r{5.2} is solvable at $x^n=0$ near $x'=0$ assuming $\partial_{x^n} \phi^n>0$, which allows us to take the positive sign of the square root to determine $\partial_{x^n} \phi^n$. This means that we can write it as
\be{nf}
\partial_{x^n}\phi = F(x,\partial_{x'}\phi ), \quad \phi(x',0)=0
\ee
with $F$ an $(n+1)$-vector valued function satisfying $F(0,0)=0$; analytic near $(0,0)$ under our analyticity assumptions. By the Cauchy–Kowalevski theorem, see \cite{Evans, Folland},  there exists an unique, in the class of analytic functions, local solution of \r{p5.1a}, \r{5.2}, near $x=0$. Then there is an analytic $\psi$ solving \r{p5}, \r{p4} near $x=0$ as well; and then we can determine $\mu$, also analytic, locally by \r{p2} as well since $\hat g_{00}(0)\not=0$. Then \r{p1} is satisfied locally as well. The jets of $g$ and $\psi^*(\mu\hat g)$ coincide near $x=0$ by the results of section~\ref{sec_jet}, therefore $g=\psi^*(\mu\hat g)$ near $x=0$ by analytic continuation. 
\end{proof}


\end{document}